\documentclass[12pt,reqno]{amsart}

\date{January 5, 2021}

\usepackage{amsmath,amsthm,amsfonts,amssymb,amsxtra,appendix,bookmark,dsfont,bm,color}
\usepackage[shortlabels]{enumitem}
\usepackage{graphicx}

\theoremstyle{plain}
\newtheorem{theorem}{Theorem}
\newtheorem{lemma}[theorem]{Lemma}

\theoremstyle{definition}

\newcounter{step} 

\usepackage{etoolbox}
\AtBeginEnvironment{proof}{\setcounter{step}{0}}

\DeclareMathOperator*{\diam}{diam}
\DeclareMathOperator{\Per}{Per}

\def\bq{\begin{eqnarray}}
\def\eq{\end{eqnarray}}
\def\bqq{\begin{align*}}
\def\eqq{\end{align*}}

\def\nn{\nonumber}

\newcommand\1{{\ensuremath {\mathds 1} }}

\def\R {\mathbb{R}}

\def\cE {\mathcal{E}}

\def\R {\mathbb{R}}

\def\d{{\, \rm d}}

\def\div{{\rm div}}
\newcommand{\Sph}{\mathbb{S}}

\title[Existence and nonexistence in the liquid drop model]{Existence and nonexistence\\ in the liquid drop model}

\author[R.L. Frank]{Rupert L. Frank}
\address[R.L. Frank]{Department of Mathematics, LMU Munich, Theresienstrasse 39, 80333 Munich, Germany, and Munich Center for Quantum Science and Technology (MCQST), Schellingstr. 4, 80799 M\"unchen, Germany, and Department of Mathematics, California Institute of Technology, Pasadena, CA 91125, USA} 
\email{rlfrank@caltech.edu}

\author[P.T. Nam]{Phan Th\`anh Nam}
\address[P.T. Nam]{Department of Mathematics, LMU Munich, Theresienstrasse 39, 80333 Munich, Germany, and Munich Center for Quantum Science and Technology (MCQST), Schellingstr. 4, 80799 M\"unchen, Germany} 
\email{nam@math.lmu.de}

\begin{document}

\begin{abstract} We revisit the liquid drop model with a general Riesz potential. Our new result is the existence of minimizers for the conjectured optimal range of parameters. We also prove a conditional uniqueness of minimizers and a nonexistence result for heavy nuclei.
\end{abstract}

\makeatletter{\renewcommand*{\@makefnmark}{}

\maketitle

\renewcommand{\thefootnote}{${}$} \footnotetext{\copyright\, 2021 by the authors. This paper may be reproduced, in its entirety, for non-commercial purposes.}

\section{Introduction}

Let $N \ge 2$, $\lambda\in (0,N)$ and $m>0$ ($\lambda$ and $m$ are not necessarily integers). For any measurable set  $\Omega \subset\R^N$, define 
$$ \cE(\Omega)=\Per \Omega + D(\Omega), \quad D(\Omega)=\frac{1}{2}\iint_{\Omega\times \Omega}\frac{\d x \d y}{|x-y|^\lambda} \,.$$
The perimeter $\Per \Omega$ is taken in the sense of De Giorgi, namely
$$ \Per \, \Omega =\sup \left\{ \int_\Omega \div F(x) \d x \,|\, F\in C_0^1(\R^3,\R^3), |F|\le 1 \right\},$$
which is simply the surface area of $\Omega$ when the boundary is smooth. We consider the minimization problem 
$$
E(m)=\inf_{|\Omega|=m} \cE(\Omega).
$$

The most important case is $\lambda=1$ in dimension $N=3$, which goes back to Gamow's liquid drop model for atomic nuclei \cite{Gamow-30}. In this case, a nucleus is thought of consisting of nucleons (protons and neutrons) in a set $\Omega\subset \R^N$. The nucleons are assumed to be concentrated with constant density, which implies that the number of nucleons is proportional to $|\Omega|$. The perimeter term in the energy functional corresponds to a surface tension, which holds the nuclei together. The second term in the energy functional corresponds to a Coulomb repulsion among the protons. Here for simplicity we have scaled all physical constants to be unity.

In the last decade, this model (for general $\lambda$ and $N$) has gained renewed interest in the mathematics literature. We refer to \cite{ChoMurTop-17} for a review and, for instance, to \cite{KnuMur-13,LuOtt-14,KnuMur-14,Julin-14,BonCri-14,FraLie-15,FFMMM-15,FraKilNam-16,Julin-17,Fra19} and references therein; see also \cite{FraNamvdB,Nam}. A variant of the problem with a constant background has also been intensely studied, see, for instance, \cite{AlbChoOtt-09,ChoPel-10,ChoPel-11,CicSpa-13,KnuMurNov-15,EmFrKo,FraLie-19} and references therein.

In principle, the two terms in $\cE(\Omega)$ are competing against each other: balls minimize the first term (by the isoperimetric inequality \cite{DeGiorgi-58}, see also \cite[Theorem 14.1]{Ma}) and maximize the second term (by the Riesz rearrangement inequality \cite{Riesz-30}, see also \cite[Theorem 3.7]{LiLo}). Thus the question about the existence of a minimizer for $E(m)$ is nontrivial. 

Clearly, the existence will depend on the parameter $m>0$. By scaling $\Omega \mapsto m^{1/N}U$ with $|U|=1$, we see that
\begin{align*}
\cE(\Omega) = m^{\frac{N-1}N} \Per U + m^{\frac{2N-\lambda}N} \ D(U) = m^{\frac{N-1}N} \Big( \Per U + m^{\frac{N+1-\lambda}N}\  D(U) \Big). 
\end{align*}
Note that $(N+1-\lambda)/N>0$. This suggests that for small $m$ the short range attraction due to the perimeter term is dominant, whereas for large $m$ the long range repulsion due to the Riesz potential is dominant. Correspondingly, we expect that there is a minimizer for small $m$ and there is no minimizer for large $m$.

In the case $\lambda=1$, $N=3$, the physics literature suggests that there is a critical volume $m_*>0$ such that balls are unique minimizers for $E(m)$ when $m\leq m_*$ and there is no minimizer when $m>m_*$ . The value $m_*$ corresponds to the threshold where the energy of a ball of volume $m$ is equal to that of two balls of mass $m/2$ each spaced infinitely far apart. It can be computed explicitly to be (see \cite{ChoPel-11,FraLie-15})
$$m_*=\frac{|B_1|\ \Per B_1}{D(B_1)} \cdot \frac{2^{1/3}-1}{1-2^{-2/3}} = 5\frac{2^{1/3}-1}{1-2^{-2/3}}\approx 3.512$$
with $B_1$ the unit ball in $\R^3$. A mathematical proof of this remains unknown. 

In the present paper, we consider the general case $N\ge 2$ and $\lambda\in (0,N)$. We define the critical volume $m_*$ to be the unique value such that 
$$
\cE\Big( \Big( \frac{m_*}{|B_1|}\Big)^{ \frac 1 N} B_1 \Big)= 2\, \cE\Big( \Big( \frac{m_*}{2|B_1|}\Big)^{1/N} B_1\Big) \,,
$$
namely,
\begin{equation} \label{eq:def-m*}
m_* = \Big( \frac{2^{1/N}-1}{1-2^{(\lambda-N)/N}} \cdot \frac{\Per B_1 }{D(B_1)}\Big)^{N/(N+1-\lambda)} |B_1|. 
\end{equation}
Here $B_1$ is the unit ball in $\R^N$ (hence, $(m/|B_1|)^{1/N} B_1$ is a ball of measure $m$). Thus, just like in the special case $\lambda=1$, $N=3$, this is the critical value where the energy of a ball of volume $m_*$ is equal to that of two balls of mass $m_*/2$, each spaced infinitely far apart, and it is natural to conjecture that $m_*$ divides the regime where minimizers are balls from the regime where there are no minimizers.

The following results were proved by Kn\"upfer and Muratov \cite{KnuMur-13,KnuMur-14}:

\begin{itemize}

\item[(a)] For every $N\ge 2$ and $\lambda \in (0,N)$, there exists a constant $m_{c_1}>0$ such that $E(m)$ has a minimizer for every $m \le m_{c_1}$. 

\item[(b)]  For every $N\ge 2$ and $\lambda \in (0,2)$, there exists a constant $m_{c_2}>0$ such that $E(m)$ has no minimizer for every $m > m_{c_2}$. 

\item[(c)] If $N=2$ and $\lambda>0$ is sufficiently small, then $m_{c_1}=m_{c_2}=m_*$ and balls are unique minimizers for $E(m)$ with $m\le m_*$. 
 
\item[(d)] if $N=2$ and $\lambda<2$, or if $3\le N \le 7$ and $\lambda<N-1$, then there exists a constant $0<m_{c_1}' \le m_{c_1}$ such that balls are unique minimizers for $E(m)$ with $m < m_{c_1}'$.  
\end{itemize}

In the most important case $\lambda=1$, $N=3$, see also \cite{LuOtt-14,FraKilNam-16} for alternative proofs of the non-existence result (b) and \cite{Julin-14} for a short proof of the uniqueness result (d). In  \cite{BonCri-14}, Bonacini and Cristoferi extended (c) and (d) to all $N\ge 2$. In \cite{FFMMM-15}, Figalli, Fusco, Maggi, Millot and Morini  extended (d) to all $N\ge 2$ and $\lambda\in (0,N)$. 

Our first new result concerns the existence in (a). Except when $\lambda>0$ is small, the existence of minimizers for $E(m)$ is known only for small $m$. In this paper, we extend the existence to what is conjectured to be the optimal range of parameters. 

\begin{theorem}\label{thm:existence} Let $N\ge 2$ and $\lambda\in (0,N)$. Then the variational problem $E(m)$ has a minimizer for every $0<m\le m_*$, where $m_*$ is defined in \eqref{eq:def-m*}. 
\end{theorem}

We will prove Theorem \ref{thm:existence} by establishing the strict binding inequality \cite{FraLie-15}
\begin{equation} \label{eq:binding-intro}
E(m) < E(m_1) + E(m-m_1), \quad \forall 0<m_1<m 
\end{equation} 
for all $m<m_*$. As a by product of our proof, we obtain the following conditional uniqueness of minimizers.   

\begin{theorem}\label{thm:uniqueness} Let $N\ge 2$ and $\lambda\in (0,N)$. If $E(m)$ has no minimizer when $m>m_*$, then balls are minimizers for $E(m)$ when $m\leq m_*$ and they are unique minimizers when $m<m_*$.
\end{theorem}

So far, the non-existence result in the sharp range $m>m_*$ is only available for $\lambda>0$ small  \cite{KnuMur-13,BonCri-14}. For larger $\lambda$ and a nonexplicit range of $m$, we have

\begin{theorem}\label{nonex} Let $N\ge 2$ and $\lambda\in (0,N)$ and $\lambda\leq 2$. Then there exists a constant $m_{c_2} \ge m_*$ such that $E(m)$ does not have a minimizer for all $m>m_{c_2}$.
\end{theorem}

 This result is due to \cite{KnuMur-13,KnuMur-14,LuOtt-14} for $\lambda<2$ and seems to be unpublished for $\lambda=2$. We will combine the methods in  \cite{FraKilNam-16} and \cite{KnuMur-13,KnuMur-14}. It is an open problem whether the nonexistence result also holds for $2<\lambda<N$ when $N\ge 3$.

\subsection*{Acknowledgements}
Partial support through U.S. National Science Foundation grants DMS-1363432 and DMS-1954995 (R.L.F.) and through the Deutsche For\-schungs\-gemeinschaft (DFG, German Research Foundation) through Germany’s Excellence Strategy EXC - 2111 - 390814868 (R.L.F., P.T.N.) is acknowledged.


\section{Existence}


In this section we prove Theorem \ref{thm:existence}. We will deduce Theorem \ref{thm:existence} from the following strict binding inequality. 


\begin{theorem}\label{thm:binding-inequality}  Let $N\ge 2$ and $\lambda\in (0,N)$. Then for every $0<m<m_*$ with $m_*$ in \eqref{eq:def-m*}, we have 
\begin{equation} \label{eq:binding}
E(m) < E(m_1) + E(m-m_1), \quad \forall 0<m_1<m. 
\end{equation} 
\end{theorem}



Thanks to \cite[Theorem 3.1]{FraLie-15}, the strict binding inequality \eqref{eq:binding} is a sufficient condition for the existence of minimizers of $E(m)$. Moreover, by \cite[Theorem 3.4]{FraLie-15}, the set $\{m>0: E(m) \text{ has a minimizer}\}$ is closed in $(0,\infty)$. Hence,   Theorem \ref{thm:binding-inequality} implies the existence of minimizers of $E(m)$ for all $0<m\le m_*$. Note that the proofs of Theorems 3.1 and 3.4 in \cite{FraLie-15} extend, without modifications, to the case $\lambda\neq 1$; see Remark 3.7 in that paper. 

We will prove the strict binding inequality using a scaling argument, based on the following key observation which uses only the isoperimetric inequality. 


\begin{lemma} \label{lem:binding} If $0<m_1<m$, then we have, with $s= m_1/m \in (0,1)$ and $B_1$ the unit ball in $\R^N$, 
$$
E(m_1) \ge s^{(2N-\lambda)/N} E(m) + (1-s^{(N+1-\lambda)/N}) s^{(N-1)/N}  \Big( \frac{m}{|B_1|}\Big)^{(N-1)/N} \Per B_1.
$$
\end{lemma}


\begin{proof} Take $\Omega \subset \R^N$ such that $|\Omega|=m_1$. Then $|s^{-1/N} \Omega|=m$, and hence
\begin{align*}
E(m) \le \mathcal E (s^{-1/N} \Omega) &= s^{-(N-1)/N}  \Per \Omega + s^{-(2N-\lambda)/N} D (\Omega)  \\
&= s^{-(2N-\lambda)/N}  \mathcal E(\Omega) -\Big( s^{-(2N-\lambda)/N}  - s^{-(N-1)/N} \Big) \Per \ \Omega.
\end{align*}
By the isoperimetric inequality
$$
\Per \ \Omega \ge \Big( \frac{m_1}{|B_1|}\Big)^{(N-1)/N} \Per B_1 = s^{(N-1)/N} \Big( \frac{m}{|B_1|}\Big)^{(N-1)/N} \Per B_1. 
$$
Thus
 \begin{align*}
E(m) \le  s^{-(2N-\lambda)/N} \mathcal E(\Omega) -  \Big( s^{-(2N-\lambda)/N}  - s^{-(N-1)/N} \Big)  s^{(N-1)/N} \Big( \frac{m}{|B_1|}\Big)^{(N-1)/N} \Per B_1.
\end{align*}
Optimizing over all $\Omega$ satisfying $|\Omega|=m_1$ we get 
 \begin{align*}
E(m) \le  s^{-(2N-\lambda)/N} E(m_1)  -  \Big( s^{-(2N-\lambda)/N}  - s^{-(N-1)/N} \Big)  s^{(N-1)/N} \Big( \frac{m}{|B_1|}\Big)^{(N-1)/N} \Per B_1
\end{align*}
which is equivalent to the desired inequality. 
\end{proof}


\begin{proof}[Proof of Theorem \ref{thm:binding-inequality}]

Take $0<m_1<m<m_*$. Denote $s=m_1/m\in (0,1)$. By Lemma \ref{lem:binding} we have
\begin{align*}
E(m_1) &\ge s^{(2N-\lambda)/N} E(m) + (1-s^{(N+1-\lambda)/N}) s^{(N-1)/N}  \Big( \frac{m}{|B_1|}\Big)^{(N-1)/N} \Per B_1,\\
E(m - m_1) &\ge (1-s)^{(2N-\lambda)/N}  E(m) \\
&\qquad + (1- (1-s)^{(N+1-\lambda)/N})(1-s)^{(N-1)/N} \Big( \frac{m}{|B_1|}\Big)^{(N-1)/N} \Per B_1.
\end{align*}
Therefore,
 \begin{align} \label{eq:binding-00}
& E(m_1)+  E(m-m_1)- E(m) \ge \Big(  s^{(2N-\lambda)/N} + (1-s)^{(2N-\lambda)/N} - 1  \Big) E(m) \nn \\
 & + \Big( (1-s^{(N+1-\lambda)/N}) s^{(N-1)/N} +  (1- (1-s)^{(N+1-\lambda)/N})(1-s)^{(N-1)/N}  \Big) \times \nn\\
 &\quad \times \Big( \frac{m}{|B_1|}\Big)^{(N-1)/N} \Per \ B_1. 
 \end{align}
 Moreover, by the variational principle, 
\begin{equation} \label{eq:em<=eball}
 E(m)\le \mathcal E\Big( \Big(\frac{m}{|B_1|}\Big)^{1/N} B_1 \Big) = \Big(\frac{m}{|B|}\Big)^{(N-1)/N} \Per B_1 +  \Big(\frac{m}{|B_1|}\Big)^{(2N-\lambda)/N} D(B_1).
\end{equation}
Inserting \eqref{eq:em<=eball} in \eqref{eq:binding-00} and using 
\begin{equation} \label{eq:s+1-s<1}
 s^{(2N-\lambda)/N} + (1-s)^{(2N-\lambda)/N} - 1<0, \quad \forall s\in (0,1),
 \end{equation}
 we find that 
  \begin{align*}
& E(m_1)+  E(m-m_1)- E(m) \\
&\ge \Big(  s^{(2N-\lambda)/N} + (1-s)^{(2N-\lambda)/N} - 1  \Big) \Big( \Big(\frac{m}{|B|}\Big)^{(N-1)/N} \Per B_1 +  \Big(\frac{m}{|B_1|}\Big)^{(2N-\lambda)/N} D(B_1) \Big)   \\
 & \quad + \Big( (1-s^{(N+1-\lambda)/N}) s^{(N-1)/N} +  (1- (1-s)^{(N+1-\lambda)/N})(1-s)^{(N-1)/N}  \Big) \times \\
 &\qquad \times \Big( \frac{m}{|B_1|}\Big)^{(N-1)/N} \Per \ B_1\\
 &= \Big( s^{(N-1)/N} +  (1-s)^{(N-1)/N} -1  \Big) \Big( \frac{m}{|B_1|}\Big)^{(N-1)/N}  \Per \ B_1 \\
 & \quad + \Big(  s^{(2N-\lambda)/N} + (1-s)^{(2N-\lambda)/N} - 1  \Big)   \Big(\frac{m}{|B_1|}\Big)^{(2N-\lambda)/N} D(B_1) \\
 &=  \Big(  s^{(2N-\lambda)/N} + (1-s)^{(2N-\lambda)/N} - 1  \Big)  \Big( \frac{m}{|B_1|}\Big)^{(N-1)/N} \Per \ B_1 \times \\
 &\quad \times \Big( \frac{D(B_1)}{\Per \ B_1}  \Big(\frac{m}{|B_1|}\Big)^{(N+1-\lambda)/N} - f(s)  \Big)
 \end{align*}
 with 
\begin{equation} \label{eq:def-f}
 f(s):= \frac{s^{(N-1)/N} +  (1-s)^{(N-1)/N} -1  }{1-s^{(2N-\lambda)/N} - (1-s)^{(2N-\lambda)/N} }. 
\end{equation}
 Using again \eqref{eq:s+1-s<1}, we find that the strict binding inequality 
 $$
 E(m_1)+  E(m-m_1)- E(m)  >0
 $$
 holds true if
\begin{equation} \label{eq:gs>m}
f(s)> \frac{D(B_1)}{\Per \ B_1}  \Big(\frac{m}{|B_1|}\Big)^{(N+1-\lambda)/N} ,\quad \forall s\in (0,1).
\end{equation}

On the other hand,  we can show that (see Lemma \ref{lem:g} below)
\begin{equation} \label{eq:gs-min}
 \min_{s\in (0,1)} f(s)= f(1/2)=  \frac{2^{1/N}-1}{1-2^{(\lambda-N)/N}}. 
\end{equation}
 Therefore, \eqref{eq:gs>m} holds true when 
 $$
 \frac{2^{1/N}-1}{1-2^{(\lambda-N)/N}} > \frac{D(B_1)}{\Per \ B_1}  \Big(\frac{m}{|B_1|}\Big)^{(N+1-\lambda)/N}
 $$
 which is equivalent to $m<m_*$. 
 \end{proof}

It remains to prove \eqref{eq:gs-min}. We have


\begin{lemma} \label{lem:g} For all $0<a<1<b<2$ and $0<s<1$ we have 
$$
\frac{s^a + (1-s)^a -1 }{2^{1-a}-1} \ge - \frac{s \log s + (1-s)\log (1-s)}{\log 2} \ge \frac{s^b + (1-s)^b -1 }{2^{1-b}-1}. 
$$
\end{lemma}

The inequality \eqref{eq:gs-min} follows from Lemma \ref{lem:g} with $a=(N-1)/N$ and $b=(2N-\lambda)/N$. 


\begin{proof} We will prove that for all $\alpha \in (0,2)$ and  all $s \in (0,1)$, 
\begin{equation} \label{eq:def-g-general}
g(s):= s^\alpha + (1-s)^\alpha -1 + \frac{2^{1-\alpha}-1} {\log 2} \Big( s \log s + (1-s)\log (1-s) \Big) \ge 0.
\end{equation}
Then the desired conclusion follows from \eqref{eq:def-g-general} and the fact that 
$2^{1-\alpha}-1>0$ if $\alpha\in (0,1)$ while $2^{1-\alpha}-1<0$ if $\alpha\in (1,2)$. 

\begin{center}
\includegraphics[scale=0.5]{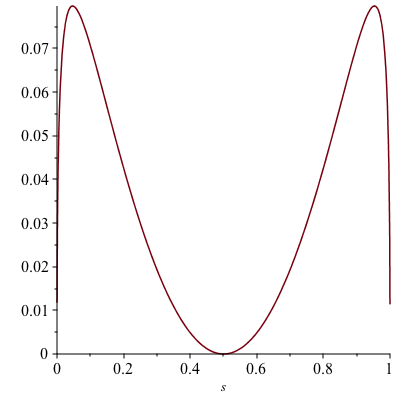} \includegraphics[scale=0.5]{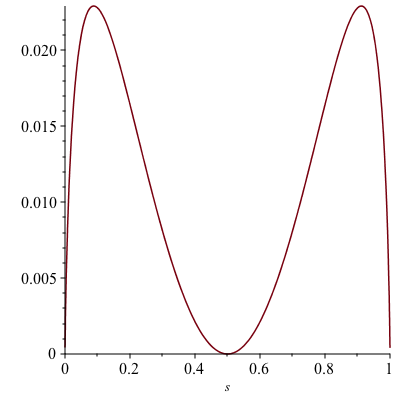} 

Figure: The function $g(s)$, $s\in (0,1)$, with $\alpha=0.5$ (left) and $\alpha=1.5$ (right)
\end{center}

\bigskip

By the symmetry $s \leftrightarrow  1-s$, it suffices to prove \eqref{eq:def-g-general} for $s\in (0,1/2]$. Also, \eqref{eq:def-g-general} is trivial when $\alpha=1$, so we will distinguish two cases $\alpha\in (0,1)$ and  $\alpha\in (1,2)$. 

\bigskip
\noindent
{\bf Case 1:} $\alpha \in (0,1)$. We have
\begin{align*}
g'(s)&= \alpha \Big(s^{\alpha-1} - (1-s)^{\alpha -1}\Big) +   \frac{2^{1-\alpha}-1} {\log 2}  \Big( \log s -\log (1-s) \Big),\\
g''(s)&= \alpha (\alpha-1) \Big(s^{\alpha-2} + (1-s)^{\alpha -2} \Big) + \frac{2^{1-\alpha}-1} {\log 2} \Big( s^{-1}+(1-s)^{-1}\Big).
\end{align*}
Define $h:(0,1/2]\to \R$ by
\begin{align*}
h(s)&:= s(1-s) g''(s) \\
&= \alpha (\alpha-1) s(1-s) \Big(s^{\alpha-2} + (1-s)^{\alpha -2} \Big) + \frac{2^{1-\alpha}-1} {\log 2} \\
&=  \alpha (\alpha-1) \Big(s^{\alpha-1} + (1-s)^{\alpha -1} - s^\alpha - (1-s)^\alpha \Big)  + \frac{2^{1-\alpha}-1} {\log 2}.
\end{align*}
Note that for all $s\in (0,1/2)$ we have
$$
h'(s)= \alpha (\alpha-1)^2 \Big( s^{\alpha-2} - (1-s)^{\alpha-2} \Big) +  \alpha^2 (1-\alpha) \Big( s^{\alpha-1} - (1-s)^{\alpha-1}\Big)> 0
$$
since 
$$s^{\alpha-2} - (1-s)^{\alpha-2}> 0, \quad (1-\alpha) ( s^{\alpha-1} - (1-s)^{\alpha-1}) >0.$$ 
Thus $h$ is strictly increasing on $(0,1/2]$. Moreover, 
$$
\lim_{s\to 0^+} h(s)=-\infty
$$
and 
\begin{align*}
h(1/2) &= \alpha (\alpha-1) 2^{1-\alpha} + \frac{2^{1-\alpha}-1} {\log 2}= 2^{1-\alpha} \Big( \frac{1}{\log 2}  - \alpha (1-\alpha) \Big) -  \frac{1}{\log 2}\\
&\ge  \Big(1 + (1-\alpha) \log 2\Big) \Big(\frac{1}{\log 2} - \alpha (1-\alpha) \Big) -\frac{1}{\log 2}\\
&= (1-\alpha)^2 \Big[ 1 - \alpha \log 2 \Big]>0
\end{align*}
since 
$$2^{1-\alpha}=e^{(1-\alpha)\log 2} \ge 1 + (1-\alpha)\log 2, \quad \alpha (1-\alpha) \le \frac{1}{4} < \log 2.$$
Thus there exists a unique value $s_1\in (0,1/2)$ (depending on $\alpha$) such that 
$$ h(s) <0\text { on }s\in (0,s_1), \quad h(s)>0 \text{ on } s\in (s_1,1/2).$$ 
Putting back the definition $h(s)= s(1-s) g''(s)$, we find that
$$ g''(s) <0\text { on }s\in (0,s_1), \quad g''(s)>0 \text{ on } s\in (s_1,1/2).$$ 

Thus $g'(s)$ is strictly decreasing on $s\in (0,s_1)$ and strictly increasing on $s\in (s_1,1/2)$. Combining with 
$$
\lim_{s\to 0^+} g'(s)= \infty, \quad g'(1/2)= 0
$$
we find that there exists a unique value $s_2\in (0,1/2)$ (depending on $\alpha$) such that
$$ g'(s) >0\text { on }s\in (0,s_2), \quad g'(s)<0 \text{ on } s\in (s_2,1/2).$$ 

Thus $g(s)$ is strictly  increasing  on $s\in (0,s_2)$ and strictly decreasing on $s\in (s_2,1/2)$. Therefore, 
$$
\inf_{s\in (0,1/2]}g(s) = \min \{ \lim_{s\to 0^+} g(s), g(1/2) \}=0. 
$$

\noindent
{\bf Case 2:} $\alpha \in (1,2)$. We can proceed similarly. To be precise, the function $h(s)=s(1-s)g''(s)$ also satisfies 
$$
h'(s)= \alpha (\alpha-1)^2 \Big( s^{\alpha-2} - (1-s)^{\alpha-2} \Big) +  \alpha^2 (1-\alpha) \Big( s^{\alpha-1} - (1-s)^{\alpha-1}\Big)> 0
$$
for all $s\in (0,1/2)$. Thus $h$ is also strictly increasing on $(0,1/2]$. Moreover, 
$$
\lim_{s\to 0^+} h(s)=\frac{2^{1-\alpha}-1}{\log 2}<0
$$
and 
\begin{align*}
h(1/2) &= \alpha (\alpha-1) 2^{1-\alpha} + \frac{2^{1-\alpha}-1} {\log 2}= \alpha (\alpha-1) 2^{1-\alpha} + \frac{ e^{(1-\alpha) \log 2} -1 }{\log 2}\\
&\ge  \alpha (\alpha-1) 2^{1-\alpha} + 1-\alpha = (\alpha-1) (\alpha 2^{1-\alpha}-1)>0. 
\end{align*}
Here we have used that $\alpha 2^{1-\alpha}>1$ for all $\alpha\in (1,2)$ (the function $q(\alpha)=\alpha 2^{1-\alpha}$ is concave on $(1,2)$ as $q''(\alpha)=2^{1-\alpha}(\log 2) ( \alpha \log 2  - 2)<0$ and $q(1)=q(2)=1$). 

Thus there exists a unique value $s_1\in (0,1/2)$ (depending on $\alpha$) such that 
$$ h(s) <0\text { on }s\in (0,s_1), \quad h(s)>0 \text{ on } s\in (s_1,1/2).$$ 
The rest is exactly the same as in Case 1. This completes the proof of Lemma \ref{lem:g}.  
\end{proof}

\section{Uniqueness}

\begin{proof}[Proof of Theorem \ref{thm:uniqueness}] {\bf Step 1.} We prove that balls are minimizers for $E(m_*)$. Assume by contradiction that balls are not minimizers for $E(m_*)$, namely
$$
E(m_*)< \cE( (m_*/|B_1|)^{1/N} B_1). 
$$
Since $m\mapsto E(m)$ and $m\mapsto \cE( (m/|B_1|)^{1/N} B_1)$ are continuous functions, there exists a constant $\delta\in (0,1)$ such that for all $m \in [m_*, m_*+\delta)$ we have 
\begin{align} \label{eq:em<eball}
E(m) &\le (1-\delta)\cE( (m/|B_1|)^{1/N} B_1) \nn\\
&\le \Big(\frac{m}{|B_1|}\Big)^{(N-1)/N} \Per B_1 +  \Big(\frac{m}{|B_1|}\Big)^{(2N-\lambda)/N} (1-\delta) D(B_1). 
\end{align}
This is similar to \eqref{eq:em<=eball}, but $D(B_1)$ is replaced by $(1-\delta) D(B_1)$. Proceeding similarly as in the proof of Theorem \ref{thm:binding-inequality} and  inserting  \eqref{eq:em<eball} (instead of \eqref{eq:em<=eball}) in \eqref{eq:binding-00}, for all $m\in [m_*, m_*+\delta)$ and $0<m_1<m$ we have
  \begin{align*}
& E(m_1)+  E(m-m_1)- E(m) \\
 &\ge  \Big(  s^{(2N-\lambda)/N} + (1-s)^{(2N-\lambda)/N} - 1  \Big)  \Big( \frac{m}{|B_1|}\Big)^{(N-1)/N} \Per \ B_1 \times \\
 &\quad \times \Big( \frac{(1-\delta) D(B_1)}{\Per \ B_1}  \Big(\frac{m}{|B_1|}\Big)^{(N+1-\lambda)/N} - f(s)  \Big)
 \end{align*}
with $s=m_1/m\in (0,1)$ and with the same function $f(s)$ in \eqref{eq:def-f}. By \eqref{eq:gs-min}, we conclude that
$$
 E(m_1)+  E(m-m_1)- E(m) >0 , \quad \forall 0<m_1<m
$$
provided that 
$$
\frac{2^{1/N}-1}{1-2^{(\lambda-N)/N}} \ge \frac{(1-\delta)D(B_1)}{\Per \ B_1}  \Big(\frac{m}{|B_1|}\Big)^{(N+1-\lambda)/N} 
$$
which is equivalent to 
$$
m\le m_*(1-\delta)^{- N/(N+1-\lambda)}. 
$$
Thus the variational problem $E(m)$ has a minimizer for all 
$$m\le \min\{ m_*+\delta, m_*(1-\delta)^{- N/(N+1-\lambda)} \}.$$
This is a contradiction to the assumption that $E(m)$ has no minimizer if $m>m_*$. Thus we conclude that balls are minimizers for $E(m_*)$. 

\bigskip
\noindent
{\bf Step 2.} Now we prove that if $m<m_*$, then balls are unique minimizers for $E(m)$. This fact follows from \cite[Theorem 2.10]{BonCri-14} which states that the set where balls are minimizers is an interval and that one has uniqueness away from the endpoint (note that this part does not require the assumption $\lambda<N-1$ which is imposed in the rest of \cite{BonCri-14}). For the reader's convenience, we provide a direct proof below.

Consider an arbitrary measurable set $\Omega\subset \R^N$ with $|\Omega|=m<m_*$. Then proceeding as in the proof of Lemma \ref{lem:binding}, we find that
\begin{equation} \label{eq:m-m*-ball-0}
\cE(\Omega) \ge s^{(2N-\lambda)/N} E(m_*) + (1-s^{(N+1-\lambda)/N}) s^{(N-1)/N}  \Big( \frac{m_*}{|B_1|}\Big)^{(N-1)/N} \Per B_1
\end{equation}
with $s=m/m_*\in (0,1)$ and the equality occurs if and only if $\Omega$ is a ball. On the other hand, we know that balls are minimizers for $E(m_*)$, namely 
$$
E(m_*) = \cE \Big( \Big(\frac{m_*}{|B_1|}\Big)^{1/N} B_1 \Big) = \Big(\frac{m_*}{|B_1|}\Big)^{(N-1)/N} \Per B_1 +  \Big(\frac{m_*}{|B_1|}\Big)^{(2N-\lambda)/N} D(B_1). 
$$
Inserting the latter equality in \eqref{eq:m-m*-ball-0}, we obtain
\begin{align*}
\cE(\Omega) &\ge s^{(2N-\lambda)/N} \Big( \Big(\frac{m_*}{|B|}\Big)^{(N-1)/N} \Per B_1 +  \Big(\frac{m_*}{|B_1|}\Big)^{(2N-\lambda)/N} D(B_1) \Big) \\
&\qquad + (1-s^{(N+1-\lambda)/N}) s^{(N-1)/N}  \Big( \frac{m_*}{|B_1|}\Big)^{(N-1)/N} \Per B_1\\
&=   \Big(\frac{m}{|B_1|}\Big)^{(N-1)/N} \Per B_1 +  \Big(\frac{m}{|B_1|}\Big)^{(2N-\lambda)/N} D(B_1) = \cE\Big( \Big( \frac{m}{B_1}\Big)^{1/N} B_1 \Big).
\end{align*}
Thus balls are minimizers for $E(m)$; moreover, if $\Omega$ is a minimizer for $E(m)$, then the equality occurs in \eqref{eq:m-m*-ball-0} and $\Omega$ is a ball. 
 \end{proof}

\section{Nonexistence}

In this section we prove Theorem \ref{nonex}. First, by extending the analysis for $\lambda=1$ in \cite{FraKilNam-16} to general $\lambda$, we have

\begin{lemma}\label{fkn}  Let $N\ge 2$ and $\lambda\in (0,N)$. Let $m>0$ be arbitrary. Let $\Omega\subset\R^N$ be a minimizer for $E(m)$. Then
	$$
	\iint_{\Omega\times \Omega} \frac{dx\,dy}{|x-y|^{\lambda-1}} \lesssim |\Omega| \,,
	$$
	with an implied constant depending only on $\lambda$ and $N$.
\end{lemma}

\begin{proof}
	For $\nu\in\Sph^{N-1}$ and $t\in\R$ we set
	$$
	\Omega^{\pm}_{\nu,t} := \Omega \cap \{ x\in\R^N:\ \pm \nu\cdot x > \pm t \} \,.
	$$
	For any $\rho\geq 0$, the set
	$$
	\Omega^+_{\nu,t} \cup \left( \Omega^-_{\nu,t} - \rho \nu\right)
	$$
	has measure $|\Omega^+_{\nu,t} \cup \left( \Omega^-_{\nu,t} - \rho \nu\right)|=|\Omega|=m$ and therefore, by minimality of $\Omega$,
	\begin{equation}
		\label{eq:fknmin}
		\cE\!\left(\Omega^+_{\nu,t} \cup \left( \Omega^-_{\nu,t} - \rho \nu\right)\right) \geq \cE(\Omega) \,.
	\end{equation}
	For any $\rho>0$, we have
	$$
	\Per \left( \Omega^+_{\nu,t} \cup \left( \Omega^-_{\nu,t} - \rho \nu\right) \right) = \Per \Omega^+_{\nu,t} + \Per \Omega^-_{\nu,t} \leq \Per \Omega + 2 \sigma(\Omega\cap \{\nu\cdot x = t\}) \,,
	$$
	where $\sigma$ denotes the induced measure on the hyperplane $\{\nu\cdot x = t\}$ and where the inequality holds for almost every $t\in\R$. 
	
	On the other hand, for any $\rho\geq 0$,
	\begin{align*}
		\iint_{ \left( \Omega^+_{\nu,t} \cup \left( \Omega^-_{\nu,t} - \rho \nu\right) \right)\times \left( \Omega^+_{\nu,t} \cup \left( \Omega^-_{\nu,t} - \rho \nu\right) \right)} \frac{dx\,dy}{|x-y|^\lambda} & = 
		\iint_{\Omega^+_{\nu,t}\times \Omega^+_{\nu,t}} \frac{dx\,dy}{|x-y|^\lambda} + \iint_{\Omega^-_{\nu,t}\times \Omega^-_{\nu,t}} \frac{dx\,dy}{|x-y|^\lambda} \\
		& \quad + 2 \iint_{\Omega^+_{\nu,t}\times \Omega^-_{\nu,t}} \frac{dx\,dy}{|x-y+\rho\nu|^\lambda} \,.
	\end{align*}
	The last double integral tends to zero as $\rho\to\infty$
	
	Inserting these facts into \eqref{eq:fknmin} and letting $\rho\to\infty$, we infer
	\begin{align*}
		& \Per \Omega + 2 \sigma(\Omega\cap \{\nu\cdot x = t\})  + \frac12 \iint_{\Omega^+_{\nu,t}\times \Omega^+_{\nu,t}} \frac{dx\,dy}{|x-y|^\lambda} + \frac12 \iint_{\Omega^-_{\nu,t}\times \Omega^-_{\nu,t}} \frac{dx\,dy}{|x-y|^\lambda} \\
		& \quad \geq \cE(\Omega) \\
		&  \quad = \Per \Omega + \frac12 \iint_{\Omega^+_{\nu,t}\times \Omega^+_{\nu,t}} \frac{dx\,dy}{|x-y|^\lambda} + \frac12 \iint_{\Omega^-_{\nu,t}\times \Omega^-_{\nu,t}} \frac{dx\,dy}{|x-y|^\lambda} + \iint_{\Omega^+_{\nu,t}\times \Omega^-_{\nu,t}} \frac{dx\,dy}{|x-y|^\lambda} \,,
	\end{align*}
	that is,
	$$
	\sigma(\Omega\cap \{\nu\cdot x = t\}) \geq \frac12 \iint_{\Omega^+_{\nu,t}\times \Omega^-_{\nu,t}} \frac{dx\,dy}{|x-y|^\lambda} \,.
	$$
	Note that the double integral here can be written as $\iint_{\Omega\times \Omega} |x-y|^{-\lambda} \1_{\{\nu\cdot x>t>\nu\cdot u\}}\,dx\,dy$. Thus, integrating the inequality with respect to $t\in\R$ gives, by Fubini's theorem,
	$$
	|\Omega| \geq \frac12 \iint_{\Omega\times \Omega} \frac{(\nu\cdot(x-y))_+}{|x-y|^{\lambda}}\,dx\,dy \,.
	$$
	Finally, we average this inequality with respect to $\nu\in\Sph^{N-1}$ and use the fact that
	$$
	\int_{\Sph^{N-1}} (\nu\cdot(x-y))_+\, \frac{d\nu}{|\Sph^{N-1}|} = c_N |x-y| \,,
	$$
	to obtain the bound in the lemma.
\end{proof}

With Lemma \ref{fkn} at hand, it is easy to finish the proof of Theorem \ref{nonex} if $\lambda\leq 1$. In fact, if $\lambda=1$, the lemma gives directly $|\Omega|^2\lesssim |\Omega|$, which is the claimed bound. If $0<\lambda<1$ we use the fact that, by a small variation of Riesz's rearrangement inequality,
$$
\iint_{\Omega\times \Omega} \frac{dx\,dy}{|x-y|^{\lambda-1}} \geq \iint_{\Omega^*\times \Omega^*} \frac{dx\,dy}{|x-y|^{\lambda-1}} \simeq |\Omega^*|^{(2N-\lambda+1)/N} = |\Omega|^{(2N-\lambda+1)/N} \,.
$$
Since $(2N-\lambda+1)/N>1$, the lemma implies once again $|\Omega|\lesssim 1$.

\medskip

It remains to deal with the case $1<\lambda\leq 2$. The key is the following bound, which, in the special case $N=3$ and $\lambda=1$ appears in \cite[Eq. (2.12)]{LuOtt-14}. The proof there extends immediately to the general case, since the analogues of \cite[Lemma 3 (ii) and Lemma 4]{LuOtt-14} hold according to \cite[Lemmas 4.1 and 4.3]{KnuMur-14}.

\begin{lemma}\label{structure} Let $N\ge 2$ and $\lambda\in (0,N)$. 	Let $m\geq \omega_N$ and let $\Omega\subset\R^N$ be a minimizer for $I(m)$. Then, for $1\leq R\leq\diam \Omega$,
	$$
	|\Omega\cap B_R(x)|\gtrsim R
	\qquad\text{for a.e.}\ x\in \Omega \,,
	$$
	with an implied constant depending only on $\lambda$ and $N$.
\end{lemma}

Here $\diam \Omega$ in the lemma is understood as the diameter of the set $\{x\in\R^N:\ |\Omega\cap B_r(x)|>0 \ \text{for all}\ r>0\}$.

\medskip

\medskip

We will use this lemma to deduce Theorem \ref{nonex} for $1<\lambda\leq 2$. If $\diam \Omega\leq 2$, then $|\Omega|\lesssim 1$ and we are done. Thus, assuming $\diam \Omega > 2$, we have, by Lemma \ref{structure}, 
\begin{align*}
	\iint_{\Omega\times \Omega} \frac{dx\,dy}{|x-y|^{\lambda-1}} 
	& = (\lambda-1) \int_0^\infty \frac{dR}{R^{\lambda}} \,|\{(x,y)\in \Omega\times \Omega:\ |x-y|<R \}| \\
	& = (\lambda-1) \int_0^\infty \frac{dR}{R^{\lambda}} \int_\Omega dx\, |\Omega\cap B_R(x)| \\
	& \geq (\lambda-1) \int_1^{\diam \Omega} \frac{dR}{R^{\lambda}} \int_\Omega dx\, |\Omega\cap B_r(x)| \\
	& \gtrsim \int_1^{\diam \Omega} \frac{dR}{R^{\lambda-1}}\, |\Omega| \,.
\end{align*}
The right side is bounded from below by a constant times $|\Omega| (\diam \Omega)^{2-\lambda}$ if $\lambda<2$ and by a constant times $|\Omega|\ln\diam \Omega$ if $\lambda=2$. Combining this lower bound on the double integral with the upper bound from Lemma \ref{fkn}, we infer in either case that $\diam \Omega\lesssim 1$, which implies $|\Omega|\lesssim 1$ and therefore concludes the proof of Theorem \ref{nonex}.

\end{document}